\documentclass[11pt]{article}

\usepackage{amsmath,amssymb,amscd,amsthm,amsfonts}
\usepackage{graphicx,subfigure}
\usepackage{hyperref}
\usepackage{dsfont}
\usepackage{color}

\newtheorem{theorem}{Theorem}[section]
\newtheorem*{theoremp}{Theorem}
\newtheorem{lemma}[theorem]{Lemma}

\newtheorem{corollary}[theorem]{Corollary}

\newtheorem*{problem}{Problem}
\newtheorem{definition}[theorem]{Definition}

\newcommand{\rr}{\mathds{R}}
\newcommand{\C}{\mathcal{C}}
\newcommand{\Z}{\mathbb{Z}}

\newcommand{\ff}{\mathcal{F}}
\newcommand{\h}{{\mathds H}}

\newcommand{\ttt}{{\mathcal T}}

\def\rr{\mathds{R}}
\DeclareMathOperator{\conv}{conv}

\DeclareMathOperator{\vol}{vol}
\DeclareMathOperator{\rank}{rank}

\DeclareMathOperator{\interior}{int}
\DeclareMathOperator{\surface}{surf}
\title{Quantitative $(p,q)$ theorems in combinatorial geometry}

\author{David Rolnick \and Pablo Sober\'on}

\begin{document}

\maketitle

\begin{abstract}
	We show quantitative versions of classic results in discrete geometry, where the size of a convex set is determined by some non-negative function.  We give versions of this kind for the selection theorem of B\'ar\'any, the existence of weak epsilon-nets for convex sets and the $(p,q)$ theorem of Alon and Kleitman.  These methods can be applied to functions such as the volume, surface area or number of points of a discrete set.  We also give general quantitative versions of the colorful Helly theorem for continuous functions.
\end{abstract}

\section{Introduction}

Helly's theorem is a central result regarding the intersection structure of convex sets.  It says that \textit{a finite family of convex sets in $\rr^d$ is intersecting if and only if every subfamily of cardinality $d+1$ is intersecting} \cite{Helly:1923wr}. Among the many generalizations and extensions of Helly's theorem (see, for instance,  \cite{Danzer:1963ug,Eckhoff:1993uy, Matousek:2002td, Wenger:2004uf, Amenta:2015tp}), a crowning achievement of combinatorial convexity is the proof of the $(p,q)$ theorem by Alon and Kleitman, answering positively a conjecture by Hadwiger and Debrunner \cite{Hadwiger:1957we}.

\begin{theoremp}[N. Alon and D. J. Kleitman \cite{Alon:1992ta}]
	Given integers $p \ge q \ge d+1$, there is a constant $c=c(p,q,d)$ such that the following statement holds.  For every finite family $\ff$ of non-empty convex sets in $\rr^d$ such that out of every $p$ elements of $\ff$ there are $q$ which are intersecting, there is a set $K$ of $c$ points that intersects every element in $\ff$.
\end{theoremp}

The $(p,q)$ theorem has a rich history of variations and generalizations as well, described in the survey \cite{Eckhoff:2003ed}.  Newer results include a version for set systems of bounded $VC$-dimension \cite{Matousek:2004cs}, colorful and fractional versions \cite{Barany:2014bp} or versions with the intersections being restricted to be in certain curves \cite{Govindarajan:2015ir}.  It should be noted that even though there are plenty of results regarding existence of $(p,q)$ theorems, finding efficient bounds for $c(p,q,d)$ remains an extremely difficult problem.

Recently, there have been developments regarding versions of Helly's theorem which work if we additionally require the condition that the intersection of the convex sets is quantitatively large.  This could mean, for example, that it contains many points from a certain discrete subset $S$ of $\rr^d$ or that it has a large volume.

The first results of this kind were given by B\'ar\'any, Katchalski and Pach \cite{Barany:1982ga, Barany:1984ed}.  They proved that \textit{given a finite family $\ff$ of convex sets in $\rr^d$, if the intersection of every $2d$ of them has volume at least one, then $\vol (\cap \ff) \ge d^{-d^2}$.}  This has been recently improved by Nasz\'odi to conclude $\vol (\cap \ff) \ge d^{-cd}$ for some absolute constant $c>0$ \cite{Naszodi:2015vi}.  The result does not hold checking subfamilies of size $2d-1$.  Optimizing over the volume guaranteed in the conclusion instead of the size of the families needed to be checked, the following version of Helly's theorem was obtained.

\begin{theoremp}[J.A. De Loera et al. \cite{DeLoera:2015wp}]
There is a constant $H(\vol, d,\varepsilon)$ such that the following holds.  If $\ff$ is a finite family of convex sets in $\rr^d$ such that all subfamilies of size at most $H(\vol, d,\varepsilon)$ have an intersection of volume at least one, then $\vol (\cap \ff) \ge 1-\epsilon$.

Moreover, for any fixed $d$ we have that $H(\vol, d, \varepsilon) = \Theta (\varepsilon^{-(d-1)/2})$.
\end{theoremp}

 The quantity $H(\vol, d, \varepsilon)$ is closely related to results of approximation of convex sets by polytopes with few facets.  Due to the lower bounds in this theorem, it is impossible to obtain a result with $\vol(\cap \ff)\ge1$ in the conclusion, regardless of the size of the subfamilies we are willing to check.
 
For discrete functions over the convex sets, such as counting the number of integer points, there are also Helly-type results.  The first of this kind was presented by Doignon.

\begin{theoremp}[J.-P. Doignon \cite{Doignon:1973ht}]
	If $\ff$ is a finite family of convex sets in $\rr^d$ such that the intersection of every subfamily of size $2^d$ contains a point with integer coordinates, then the intersection of $\ff$ contains a point with integer coordinates.
\end{theoremp}

This result was rediscovered several times \cite{Bell:1977tm, Scarf:1977va, Hoffman:1979ix}.  Doignon's theorem has a quantitative generalization \cite{Aliev:2014va}.  Aliev et. al. showed that there is a constant $c(k,d)$ such that \textit{for any finite family of convex sets in $\rr^d$ such that the intersection of every $c(k,d)$ of them contains at least $k$ points with integer coordinates, the intersection of the whole family contains at least $k$ points with integer coordinates}.

Recall that a set $S \subset \rr^d$ is said to be \textit{discrete} if every point $x \in \rr^d$ has a neighborhood such that $x$ is the only point of $S$ within it.  For any discrete set $S$, we can consider versions of Helly's theorem for $S$, and in particular can consider quantitative Helly numbers, defined as follows.

\begin{definition}
	For a discrete set $S \subset \rr^d$, we define the \emph{$k$-quantitative Helly number $\h_S(k)$} as the smallest number for which the following statement holds.  For any finite family $\ff$ of convex sets in $\rr^d$ such that for every subfamily $\ff'$ of cardinality $\h_S(k)$ or less we know that $\cap \ff'$ has at least $k$ points of $S$, then $\cap \ff$ contains at least $k$ points of $S$.  If no such number exists, we consider $\h_S(k) = \infty$.
\end{definition}

An example of a large family of sets $S$ with finite $\h_S(k)$ for all $k$ is the following.  Let $L$ be a lattice in $\rr^d$ and $L_1, L_2, \ldots, L_m$ be sublattices.  If $S = L \setminus (L_1 \cup \ldots \cup L_m)$, then $\h_S(k) \le (2^{m+1}k+1)^{\rank L}$ \cite[Thm. 1.9]{DeLoera:2015wp}.  For the case where $S$ is the integer lattice, better bounds are found in \cite{Aliev:2014va}.

The aim of this paper is to show a general quantitative version of the $(p,q)$ theorem - in particular, one that captures the behavior of continuous functions as well as discrete functions over the convex sets in $\rr^d$.  In Theorem \ref{theorem-general-quantitative-pq} we show a general set of conditions that a function $f:\mathcal{C}_d \to \rr^+ \cup \{\infty\}$ must satisfy in order to give a $(p,q)$ theorem, where $\mathcal{C}_d$ refers to the family of convex sets in $\rr^d$.  In particular, we prove that the functions $f(K) = \vol (K)$, the surface area, and $f(K) = |K \cap S|$ satisfy the conditions, where $S$ is any discrete set with finite Helly numbers.  Namely, we obtain the following corollaries.

\begin{corollary}[Quantitative discrete $(p,q)$ theorem]\label{theorem-quant-pq}
Let $S$ be a discrete subset of $\rr^d$ and $k$ be an integer such that $\h_S(k)$ is finite.  Let $p \ge q \ge \h_S(k)$ be positive integers.  Then, there is a constant $c= c(p,q,k,S)$ such that the following is true.  For any finite family $\ff$ of convex sets in $\rr^d$, each containing at least $k$ points of $S$ and such that for every $p$ sets in $\ff$ there are $q$ whose intersection contains $k$ points of $S$, we can find $c$ sets $K_1, K_2, \ldots, K_c$, each containing $k$ points of $S$, such that every $F \in \ff$ contains at least one $K_i$.
\end{corollary}

\begin{corollary}[Volumetric $(p,q)$ theorem]\label{theorem-vol-pq}
Given a positive integer $d$ and $1>\varepsilon > 0$, there is a constant $n(d, \varepsilon)$ such that the following holds.  Let $p \ge q \ge n(d, \varepsilon)$ be positive integers.  Then, there is a constant $c_{\vol}= c_{\vol}(p,q,d,\varepsilon)$ such that the following is true.  For any finite family $\ff$ of convex sets in $\rr^d$, each of volume at least one and such that for every $p$ sets in $\ff$ there are $q$ whose intersection has volume at least one, we can find $c_{\vol}$ convex sets $C_1, C_2, \ldots, C_{c_{\vol}}$, each of volume $1- \varepsilon$, such that every $F \in \ff$ contains at least one $C_i$.
\end{corollary}

\begin{corollary}[Surface area $(p,q)$ theorem]\label{theorem-surface}
Given a convex set $K \subset \rr^d$, let $\surface (K)$ be its surface area.  Given a positive integer $d$ and $1>\varepsilon > 0$, there is a constant $n_2(d, \varepsilon)$ such that the following holds.  Let $p \ge q \ge n_2(d, \varepsilon)$ be positive integers.  Then, there is a constant $c_{\surface}= c_{\surface}(p,q,d,\varepsilon)$ such that the following is true.  For any finite family $\ff$ of convex sets in $\rr^d$, each of surface area at least one and such that for every $p$ sets in $\ff$ there are $q$ whose intersection has surface area at least one, we can find $c_{\surface}$ convex sets $C_1, C_2, \ldots, C_{c_{\surface}}$, each of surface area $1- \varepsilon$, such that every $F \in \ff$ contains at least one $C_i$.
\end{corollary}
For every fixed $d$, the quantity $n(d,\varepsilon)$ of Corollary \ref{theorem-vol-pq} is $O \left(\varepsilon^{-(d^2-1)/4}\right)$.  For $d=1$, it is a simple exercise to show that the classic $(p,q)$ theorem implies Corollary \ref{theorem-vol-pq}, with $\vol(C_i) = 1$ for all $i$.  For the quantity $n_2(d,\varepsilon)$ in Corollary \ref{theorem-surface}, we do not have explicit asymptotic bounds.

The case $k=1$, $S=\mathbb{Z}^d$ of Corollary \ref{theorem-quant-pq} has already been proven \cite{Alon:1995fs}.  Moreover, a surprising result is an improvement by B\'ar\'any and Matou\v{s}ek \cite{Barany:2003wg}.  They showed that the condition needed on $p,q$ is only $p \ge q \ge d+1$, while $\h_{\mathbb{Z}^d}(1) = 2^d$.  However, it seems that one of their key ingredients \cite[Corollary 2.2]{Barany:2003wg} fails when $k \ge 2$.  The fractional Helly theorems of Averkov and Weismantel \cite{Averkov:2010tv} and the methods below show that we can relax the condition to $p \ge q \ge d+1$ if $k =1$ as long as $\h_S(1)$ is finite.  Moreover, if $k=1$, the condition of $S$ being discrete can be replaced by asking $S$ to be closed. 

A key step for Corollaries \ref{theorem-vol-pq} and \ref{theorem-surface} is the existence of a ``floating body''.  These floating bodies are extensions for arbitrary functions of the definition introduced in \cite{Schutt:1990tj}, where they are developed for the volume.  This concept allows us to prove general forms of the fractional Helly theorem.  It should be pointed out that these methods do not work for all functions, in particular the diameter does not satisfy the required conditions.  The floating bodies we discuss also allow us to prove general quantitative versions of the colorful Helly theorem by Lov\'asz, whose proof appeared in a paper by B\'ar\'any \cite{Barany:1982va}.  In particular, they work for the surface area and volume.  The case of the volume was already known \cite[Thm. 1.6]{DeLoera:2015wp}, with an essentially different proof.

The paper is organized as follows.  In section \ref{section-strong} we state the general $(p,q)$ theorem, the conditions needed for it and prove the result.  In section \ref{section-fractional-both} we give exhibit a class of functions that satisfies the conditions stated in section \ref{section-strong}.  This includes volume, the surface area and the number of points from a discrete set.  Using the same methods of section \ref{section-fractional-both}, we provide a general set of properties needed to prove quantitative version of the colorful Helly theorem in section \ref{section-colorful-volumetric}.  In section \ref{section-remarks} we include some remarks and open problems.

\section{General quantitative $(p,q)$ theorem}\label{section-strong}

Let $\rr^+ = \{a \in \rr: a \ge 0\}$ and $\mathcal{C}_d$ be the family of convex sets in $\rr^d$.  We wish to determine which functions $f:\C_d \to \rr^+ \cup \{\infty\}$ allow for a quantitative $(p,q)$ theorem.  We say that a function $f:\C_d \to \rr^+ \cup \{\infty\}$ is monotone if $A \subset B$ implies $f(A) \le f(B)$.  During the rest of the paper we assume that the functions we work with are monotone.
The main three ingredients we need are that $f$ can be approximated by circumscribed polyhedra, that $f$ can be approximated by inscribed polytopes and a fractional Helly theorem. We define these three porperties below.  The first two conditions are directly related to the approximation of convex sets by polytopes, we recommend the surveys \cite{Bronstein:2008dn, Gruber:1993gj} for this subject.

Regarding the existence of approximations by circumscribed polyhedra, we can state a condition in terms of Helly-type theorems.

\begin{definition}\label{definition-helly}
	We say $f:\C_d \to \rr^+ \cup \{\infty\}$ \emph{admits a Helly theorem} if given $1>\varepsilon>0$ there is an integer $H(f,d,\varepsilon)$ such that for any finite family $\ff$ and $\lambda >0$, if every subfamily $\ff'$ of cardinality at most $H(f,d,\varepsilon)$ satifies $f(\cap \ff') \ge \lambda$, then $f(\cap \ff) \ge (1-\varepsilon)\lambda$.
\end{definition}

It was shown in \cite{Amenta:2015tp} that those functions which allow for approximations by circumscribed polyhedra always admit a Helly theorem.  Namely, it is sufficient to know that for every convex set $K$ with $f(K)<\infty$, there is a polyhedron $P$ such that
\begin{itemize}
	\item $K \subset P$,
	\item $P$ has a bounded number of facets in terms of $d, \varepsilon$, and
	\item $f(K) \ge (1-\varepsilon)f(P)$.
\end{itemize}

If $f$ is the indicator of the property \textit{``being non-empty''} in Definition \ref{definition-helly}, we recreate Helly's original result.  The results needed for the volumetric version are stated precisely in \cite{Anonymous:HS2Q-kvJ}, which give approximation estimates based on the Nikodym metric.  For the surface area, it suffices to use a compactness argument similar to that we will use in section \ref{section-fractional-both}.  For the indicator of \textit{``containing a set of $k$ points from a set $S$''}, the definition of $\h_S(k)$ is precisely the Helly theorem (with $\varepsilon = 0$).  Another equivalent formulation for this property is with the function $f(\cdot) = |S \cap \cdot|$ and the hidden constant being $k$, and $\varepsilon =0$.  In addition to the work in \cite{DeLoera:2015wp}, previous quantitative Helly-type results for continuous functions existed with contraction and expansion of convex sets \cite{Langberg:2009go}.

The next ingredient is that $f$ can be approximated by inscribed polytopes.

\begin{definition}\label{definition-caratheodory}
	We say $f:\C_d \to \rr^+\cup \{\infty\}$ \emph{can be approximated by inscribed polytopes} if for every $1> \varepsilon > 0$ there is a constant $C(f,d,\varepsilon)$ such that for every $\infty > \lambda \ge 0$ and every convex set $K \subset \rr^d$ with $f(K) \ge \lambda$, there is a polytope $P \subset \rr^d$ such that
	\begin{itemize}
		\item $P \subset K$,
		\item $P$ has at most $C(f,d,\varepsilon)$ vertices, and
		\item $f(P) \ge (1-\varepsilon) \lambda$.
	\end{itemize}
\end{definition}

For example, very precise estimates exist for $C(\vol, d, \varepsilon)$ \cite{Gordon:1995kn, Gordon:1997wi}, giving $\left(\frac{c d}{\varepsilon}\right)^{(d-1)/2}$ for some absolute constants $c$ for an upper and lower bound.  For the indicator of \textit{``having at least $k$ points of $S$''}, one gets trivially $C(f,d,0) = k$.  For the surface area, again a compactness argument as in section \ref{section-fractional-both} is sufficient, although it does not yield explicit bounds.

The methods in \cite{DeLoera:2015wp} also show that this kind of approximation is closely linked to the existence of colorful Carath\'eodory theorems, as in \cite{Barany:1982va}.  This comes as no surprise since Carath\'eodory's theorem is the natural dual of Helly's theorem, and these approximations are duals to the ones needed for Helly-type theorems.  We will be using the approximations as above in combination with the following result.

\begin{theoremp}[Quantitative colorful Carath\'eodory theorem; De Loera et al. \cite{DeLoera:2015wp}]
	Let $S$ be a set of $k$ points in $\rr^d$, and $n = \max\{kd, d+1\}$.  Then, for any $n$ sets $X_1, \ldots, X_n$ such that $S \subset \conv (X_i)$ for all $i$, there is a choice of points $x_1 \in X_1, \ldots, x_n \in X_n$ such that $S \subset \conv\{x_1, \ldots, x_n\}$.
\end{theoremp}

The colorful Carath\'eodory theorem comes when we use the theorem above with the polytope from Definition \ref{definition-caratheodory} when $f(\cdot)$ is the indicator of the property \textit{``contains the origin''} (note that $C(f,d,0) = 1$ in this case).  This has also been studied for instance for $f(K) = \max \{\lambda \ge 0 : B_{\vec{0}}(\lambda) \subset K\}$, where $B_{\vec{0}}(\lambda)$ is the Euclidean ball of radius $\lambda$ centered at the origin \cite{Kirkpatrick:1992ex, DeLoera:2015wp}.  These are often referred to as quantitative Steinitz results, which were also pioneered by B\'ar\'any, Kathalski and Pach \cite{Barany:1982ga}.  If one changes the unit ball by an arbitrary centrally symmetric set, approximations reducing the Banach-Mazur distance, such as \cite{Barvinok:2014fx}, come into play.

The final ingredient is a fractional Helly theorem.

\begin{definition}
	We say $f:\C_d \to \rr^+\cup \{\infty\}$ \emph{admits a fractional Helly theorem} if, given $1>\varepsilon>0$, there is an integer $F=F(f,d,\varepsilon)$ such that for any $\alpha>0$ there is a $\beta = \beta (\alpha, f, d, \varepsilon) >0$ such that the following holds.  Given any finite family $\ff$ of convex sets in $\rr^d$, if there are at least $\alpha {{|\ff|}\choose{F}}$ subsets $A \subset \ff$ of cardinality $F$ that satisfy $f(\cap A) \ge 1$, there is a subset $\ff' \subset \ff$ of cardinality at least $\beta |\ff|$ such that $f(\cap \ff') \ge 1-\varepsilon$.
\end{definition}

We should emphasize that it is not immediate that a Helly theorem implies a fractional Helly theorem, or vice-versa.  In particular, for the volumetric version of the $(p,q)$ theorem we present, the fractional Helly we prove has a larger bound for $F(\vol, d, \varepsilon)$ than $H(\vol, d, \varepsilon)$.

Given a finite family of convex sets in $\rr^d$ of volume at least one, if we make a simplicial complex by taking a vertex for each set and a face for each subfamily whose intersection has volume at least one, there is no indication that the resulting complex is $d'$-collapsible or even $d'$-Leray for some $d'$ depending on $d$.  In other words, the machinery developed by Gil Kalai to obtain fractional Helly results cannot be applied directly \cite{Kalai:1984isa, Kalai:1986ho}.  For a deeper discussion of the relation of topological properties of nerve complexes and Helly theorems, we recommend \cite{Tancer:2013iz}.

The most compelling cases are those for which $F(f,d,\varepsilon) \neq H(f,d,\varepsilon)$.  One instance where this can be observed is when $f$ is the indicator of the function \textit{``contains a point with integer coordinates''}, where B\'ar\'any and Matou\v{s}ek showed that the fractional Helly number is $d+1$, while the Helly number is $2^d$ \cite{Barany:2003wg}.  Other examples of this phenomenon include fractional Helly numbers for the property of \textit{``having a hyperplane transversal''} \cite{Alon:1995fs}, or fractional Helly numbers for set systems with bounded $VC$-dimension \cite{Matousek:2004cs}.  In both cases there is no Helly theorem but there is a fractional version.  We show a general way to obtain fractional Helly theorems, but the restrictions for the function increase.

For each of the definitions mentioned, we say $f$ satisfies a \textit{sharp} version of the result if it holds for $\varepsilon = 0$.  For instance, Doignon's theorem can be restated as $H(|\cdot \cap \Z^d|, d, 0) = 2^d$, with the hidden $\lambda$ being one.  Note that given a discrete set $S$, the function $f(K) = |S\cap K|$ and $\lambda = k$ would give the quantitative results for \textit{``having at least $k$ points of $S$''}.  This also sheds some light on why Helly-type results for functions that vary continuously such as the volume or diameter have a necessary loss $\varepsilon$ while discrete functions, such as $f(K) = |S\cap K|$ do not.  For $f(K) = |S\cap K|$, it suffices for $\varepsilon < \frac{1}{k}$ to get that $f(K) \ge (1- \varepsilon)k$ implies $f(K) \ge k$, so $n(f,d,\varepsilon) = n(f,d,0)$, so small values of $\varepsilon$ are of no consequence.

\begin{theorem}[Quantitative $(p,q)$ theorem]\label{theorem-general-quantitative-pq}
	Let $f:\C_d \to \rr^+\cup \{\infty\}$ be a monotone function that admits a Helly theorem, approximations by inscribed polytopes and a fractional Helly theorem.  Then, given $p \ge q \ge F(f,d,\frac{\varepsilon}{2})$, there is a constant $c=c(f,p,q,d,\varepsilon)$ such that for any finite family $\ff$ of convex sets in $\rr^d$, if out of every $p$ elements of $\ff$ there is a $q$-tuple $A$ such that $f(\cap A) \ge 1$, then there is a family $\ttt$ of at most $c$ convex sets $K_1, K_2, \ldots, K_c$ such that $f(K_i) \ge 1- \varepsilon$ for each $i$ and every set in $\ff$ contains at least one set in $\ttt$.
	
	Moreover, if $f$ admits sharp versions of the results mentioned, then we may take $\varepsilon=0$ in this theorem as well.
\end{theorem}

The constant $\frac{\varepsilon}2$ is completely arbitrary, and indeed any $\varepsilon' = \gamma \varepsilon$ for some constant $1 > \gamma >0$ would work (affecting the value of $c$, of course).  However, with the constant shown we get corollaries \ref{theorem-quant-pq}, \ref{theorem-vol-pq} and \ref{theorem-surface}.

In order to prove Theorem \ref{theorem-general-quantitative-pq}, we need new interpretations of results regarding the combinatorial structure of sets of points in $\rr^d$.  That is, given a family $S$ of points in $\rr^d$, we are interested in the combinatorial properties of the family of the convex hulls of subsets of $S$.  In particular, we need versions of Tverberg's theorem, B\'ar\'any's selection theorem and the existence of weak $\varepsilon$-nets for convex hulls.

These are classic results in combinatorial geometry.  As we present them we mention some references for the interested reader.  Each of these results could be seen as a counterpart for a Helly-type theorem.  Indeed, both Helly's theorem and Tverberg's theorem aim to find intersecting families of convex sets. Both the fractional Helly theorem and B\'ar\'any's selection theorem aim to find a positive-fraction subfamily which is intersecting.  Finally, both the $(p,q)$ theorem and the existence of weak $\varepsilon$-nets aim to bound the piercing number of a family of convex sets.

We start with Tverberg's theorem.

\begin{definition}
	We say $f:\C_d \to \rr^+ \cup \{\infty\}$ \emph{admits a Tverberg theorem} if, given $1>\varepsilon>0$, there is an integer $T(f,d,\varepsilon)$ such that for any positive integer $m$ and any real $\lambda >0$, given a family $\ttt = \{T_1, T_2, \ldots, T_n\}$ of $n=(m-1)T(f,d,\varepsilon)+1$ convex sets $T_i \subset \rr^d$ such that $f(T_i) \ge \lambda$ for all $i$, there is a partition of $\ttt$ into $m$ parts $A_1, A_2, \ldots, A_m$ such that
	\[
	f\left( \bigcap_{j=1}^m \conv (\cup A_j)\right) \ge \lambda (1-\varepsilon)
	\]
\end{definition}

We should stress that the relevant feature of the definition above is that the number of sets needed to obtain a Tverberg partition is linear on the number of parts.  That is the key that allows the next results to hold.

The classic result was proven by Tverberg in 1966 \cite{Tverberg:1966tb}, where $f$ is the indicator function of \textit{``being non-empty''}, and $T(f,d,0) = d+1$.  The case $m=2$ is known as Radon's lemma \cite{Radon:1921vh}.  Other quantitative versions of Tverberg's theorem exist \cite{DeLoera:2015wp}, but they are essentially different from the version presented here.  Finding exact bounds for the case when $f$ is the indicator function of \textit{``containing a point with integer coordinates''} is a surprisingly resistant problem, even for $m=2$ \cite{Jamison:1981wz}.  We show that a quantitative Helly and approximations by inscribed polytopes imply a Tverberg theorem.

\begin{theorem}[Quantitative Tverberg theorem]
	Let $f:\C_d \to \rr^+\cup \{\infty\}$ be a monotone function that admits a Helly theorem and approximations by inscribed polytopes.  Then, $f$ admits a Tverberg theorem.
\end{theorem}

\begin{proof}
	In order to prove the result mentioned, we will show the following.  Given $\varepsilon_1, \varepsilon_2 \in (0,1)$ and a positive integer $m$, take 
	\[n = (m-1)d\cdot H(f,d,\varepsilon_1)C(f,d,\varepsilon_2)+1.\]
	Then, given a family $\ttt = \{T_1, T_2, \ldots, T_n\}$ of sets such that $f(T_i) \ge \lambda$ for all $i$, there is a partition of $\ttt$ into $m$ parts $A_1, A_2, \ldots, A_m$ such that
	\[
	f\left( \bigcap_{j=1}^m \conv (\cup A_j) \right) \ge \lambda (1-\varepsilon_1)(1-\varepsilon_2).
	\]
	This will imply that $f$ admits a Tverberg theorem.  In particular $T(f,d,\varepsilon) \le d \cdot H(f,d,\frac{\varepsilon}2) C(f,d,\frac{\varepsilon}2)$.
	
	Given a family $\ttt$ as above, consider the family 
	\[
	\ff = \{F: F= \conv(\cup \ttt'), \ttt' \subset \ttt, |\ttt'| = (m-1)d\cdot [H(f,d,\varepsilon_1)-1]C(f,d,\varepsilon_2)+1\}
	\]
	Note that each $F \in \ff$ is missing at most $(m-1)d \cdot C(f,d,\varepsilon_2)$ sets of $\ttt$.  Thus, any $H(f,d,\varepsilon_1)$ of them have at least one $T_i$ in common, which means that their intersection has size at least $\lambda$ under $f$.  By the definition of $H(f,d,\varepsilon_1)$, there is a convex set $T_0$ with $f(T_0)\ge (1- \varepsilon_1)\lambda$ contained in $\cap \ff$.
	
	Every closed halfspace that contains a point of $T_0$ also contains points of at least $(m-1)d\cdot C(f,d,\varepsilon_2)+1$ sets of $\ttt$.  If this was not the case, there would be a a point $p \in T_0$ and closed halfspace containing $p$ and points of at most $(m-1)d\cdot C(f,d,\varepsilon_2)$ sets of $\ttt$.  This would contradict the fact that $p \in T_0 \subset \cap \ff$.  In particular, we have $T_0 \subset \conv (\cup \ttt)$.
	
	Now we construct $A_1, A_2, \ldots, A_m$ inductively.  By the definition of $C(f,d,\varepsilon_2)$, there we can find a set $P \subset T_0$ of cardinality at most $C(f,d,\varepsilon_2)$ such that $f(\conv(P)) \ge f(T_0) (1-\varepsilon_2) \ge \lambda(1-\varepsilon_1)(1-\varepsilon_2)$.  By the observation above, $P \subset \conv (\cup \ttt)$.
	
	If $C(f,d,\varepsilon_2) > 1$, by the quantitative Carath\'eodory theorem there is a set $U_1 \subset \cup \ttt$ of cardinality at most $d \cdot C(f,d,\varepsilon_2)$ such that $P \subset \conv U_1$.  For each point of $U_1$, let us take one set of $\ttt$ that contains it, and form $A_1$ in this way.  Note that $P \subset \conv (\cup A_1)$ and $|A_1| \le d \cdot C(f,d,\varepsilon_2)$.  By removing $A_1$ from $\ttt$, we have that every closed halfspace that contains a point of $T_0$ also contains points of at least $(m-2)d\cdot C(f,d,\varepsilon_2)+1$ sets of what is left of $\ttt$.  Thus, we can continue this process and generate the sets $A_1, A_2, \ldots, A_m$.  In the end,
	\[
	P \subset \bigcap_{j=1}^m \conv (\cup A_j).
	\]
	Thus, the same holds for $\conv(P)$, and since $f$ is monotone we obtain the result.
	
	If $C(f,d,\varepsilon_2) = 1$ we use the same inductive construction, but we have to take an additional precaution.  Let $\{p\}= T_0$.  When we construct $A_1$, we can take it to be minimal in cardinality.  Thus, if $|A_1| = d+1$, we have that $p \in \interior (A_1)$.  This means that every closed halfspace that contains $p$ in its boundary contains at most $d$ points of $A_1$.  In other words, we can keep the reduction process as before.
\end{proof}

Now we move on to the selection theorem.  Our goal is to show that allowing approximations by inscribed polytopes and a Tverberg theorem is enough to obtain the result below.

\begin{definition}
	Let $f:\C_d \to \rr^+ \cup \{\infty\}$ be a monotone function that can be approximated by inscribed polytopes and let $r(f,d,\varepsilon) = \max \{d \cdot C(f,d,\frac{\varepsilon}2), d+1\}$.  We say $f$ \emph{admits a selection theorem} if for any $1>\varepsilon>0$ there is a constant $\rho(f,d,\varepsilon)$ such that for any $\lambda >0$ and any finite family $\ttt$ of convex sets in $\rr^d$, if $f(T) \ge \lambda$ for all $T \in \ttt$, there is a convex set $T_0$ with $f(T_0) \ge \lambda (1-\varepsilon)$ contained in the convex hull of the union at least $\rho {{|\ttt|}\choose{r(f,d,\varepsilon)}}$ subsets $A \subset \ttt$ of cardinality $r(f,d,\varepsilon)$.
	
\end{definition}

The original result by B\'ar\'any \cite{Barany:1982va} is the case when $f$ is the indicator function of the property \textit{``is non-empty''}.  Even in this case, finding the optimal value of $\rho$ is a remarkably difficult problem.  The current best lower bound is $\rho(f,d,0) \ge \frac{1}{(d+1)!}$ \cite{Gromov:2010eb}.  We recommend \cite{Karasev:2012bj, Pach:1998vx} for further extensions and references of related results.

\begin{theorem}[Quantitative selection theorem]
	Let $f:\C_d \to \rr^+\cup \{\infty\}$ be a monotone 	function that admits a Tverberg theorem and approximations by inscribed polytopes.  Then, $f$ admits a selection theorem.
\end{theorem}

\begin{proof}
	Take $\ttt$ a finite family of convex sets in $\rr^d$ with $f(T) \ge \lambda$ for all $T \in \ttt$.  Since $f$ admits a Tverberg theorem, there is a convex set $U_0$ with $f(U_0) \ge (1- \frac{\varepsilon}2)\lambda$ and partition $P$ of $\ttt$ into $\frac{|\ttt|-1}{T(f,d,\frac{\varepsilon}2)}+1$ parts such that the convex hull of the union of each part contains $U_0$.
	
	 By the definition of $C(f,d,\frac{\varepsilon}2)$, there is a polytope $T_0 \subset U_0$ with at most $C(f,d,\frac{\varepsilon}2)$ vertices such that $f(T_0) \ge (1-\frac{\varepsilon}2)^2\lambda \ge (1-\varepsilon) \lambda$.  Color each part in $P$ with a different color.  Since the convex hull of the union of each part contains $T_0$, by the quantitative colorful Carath\'eodory theorem, for each $r(f,d,\varepsilon) = \max \{d \cdot C(f,d,\frac{\varepsilon}2), d+1\}$ colors, there is a rainbow choice of subsets of $\ttt$ such that the convex hulls of the sets contains $T_0$.  In other words, the number of $r(f,d,\varepsilon)$-tuples satisfying the desired conditions is at least
	 
	\[
	{{\frac{|\ttt|-1}{T(f,d,\frac{\varepsilon}2)}+1}\choose{r(f,d,\varepsilon)}} \ge \left(\frac{1}{T(f,d,\frac{\varepsilon}2)}\right)^{r(f,d,\varepsilon)}{{|\ttt|}\choose{r(f,d,\varepsilon)}}.
	\]
\end{proof}

Using a selection theorem, we can show the existence of weak $\varepsilon$-nets for convex sets, similarly to \cite{Alon:1992ek}.  The goal now is, given a finite family $\ttt$ of \textit{large} convex sets under $f$, to bound the quantitative equivalent of the piercing number for the family of sets generated by taking the convex hull of the union of \textit{many} subsets of $\ttt$.  We say a set $\ttt$ \textit{pierces} a family of sets $\ff$ if every set in $\ff$ contains a set in $\ttt$.  We use the selection theorem to construct a piercing set following a greedy algorithm.

\begin{definition}\label{definition-weak-net-quant}
	Given $f:\C_d \to \rr^+ \cup \{\infty\}$ and $1>\varepsilon' > 0$, we say $f$ \emph{admits bounded weak $\varepsilon'$-nets} for convex sets if for every $1>\varepsilon >0$ there is a constant $m=m(f,d,\varepsilon, \varepsilon')$ such that for every $\lambda >0$ and every finite family $\ttt$ of convex sets in $\rr^d$ with $f(T) \ge \lambda$ for all $T \in \ttt$, there is a family $\ttt_0$ of at most $m$ convex sets in $\rr^d$ with $f(T_0) \ge (1-\varepsilon)\lambda$ for all $T_0 \in \ttt_0$ such that for every subset $A \subset \ttt$ of cardinality at least $\varepsilon' |\ttt|$, there is at least one set of $\ttt_0$ contained in $\conv (\cup A)$.
	
	We say that $\ttt_0$ is a $\varepsilon'$-weak net for the pair $(\ttt, \varepsilon)$ if it satisfies the conditions above.
\end{definition}

\begin{theorem}
	If $f:\C_d \to \rr^+ \cup \{\infty\}$ is monotone, can be approximated by inscribed polytopes and admits a selection theorem, then for any $1>\varepsilon'>0$, it admits bounded weak $\varepsilon'$-nets for convex sets.
\end{theorem}

\begin{proof}
	Given $\varepsilon >0$, $\lambda >0$ and a finite family $\ttt$ of convex sets in $\rr^d$ with $f(T) \ge \lambda$ for all $T \in \ttt$, we will construct inductively a weak $\varepsilon'$-net for the pair $(\ttt, \varepsilon)$.
	
Let $r(f,d,\varepsilon) = \max\{d \cdot C(f,d,\frac{\varepsilon}2),d+1\}$.	In order to construct a weak $\varepsilon'$-net $\ttt_0$, we start with $\ttt_0$ being the empty set.  Let $R$ be the number of $r(f,d,\varepsilon)$-tuples $B \subset \ttt$ such that $\conv (\cup B)$ contains no set of $\ttt_0$.  If there is a subset  $A \subset \ttt$ of cardinality at least $\varepsilon' |\ttt|$ such that the convex hull of its union contains no set of $\ttt_0$, then we can apply the selection theorem to $A$ and find a convex set $T_0$ with $f(T) \ge \lambda (1-\varepsilon)$ that is contained in the convex hull of the union of at least
	\begin{eqnarray*}
		\rho(f,d,\varepsilon) {{|A|}\choose{r(f,d,\varepsilon)}} & \ge & \rho(f,d,\varepsilon) {{\varepsilon' |\ttt|}\choose{r(f,d,\varepsilon)}} \\
		&\ge & \rho(f,d,\varepsilon)(\varepsilon')^{r(f,d,\varepsilon)} {{|\ttt|}\choose{r(f,d,\varepsilon)}}
	\end{eqnarray*}
	sets $B \subset A$ of cardinality $r(f,d,\varepsilon)$.  Thus, by adding $T$ to $\ttt_0$, we are effectively reducing $R$ by the quantity above.  Thus, this cannot be done more than $\left(\rho(f,d,\varepsilon)(\varepsilon')^{r(f,d,\varepsilon)}\right)^{-1}$ times, which gives the desired result.
\end{proof}

With the existence of weak $\varepsilon$-nets, we are ready to prove the quantitative version of the $(p,q)$ theorem.  For this purpose, we need the following definitions.

Given a monotone function $f:\C_d \to \rr^+ \cup \{\infty\}$, we define
\[
\C_d{(f,\lambda)} = \{K \in \C_d : f(K) \ge \lambda \}.
\]

Given $\lambda >0$ and a finite family $\ff \subset \C_d$, we define

\begin{definition}
	Given $\lambda >0$ and a finite family $\ff \subset \C_d$, we define the \emph{$f$-transversal number} $\tau_{(f,\lambda)}(\ff)$ as the minimum $\sum_{C\in \C_d(f,\lambda)}w(C)$ over all functions $w:\C_d(f,\lambda) \to \{0,1\}$ such that
	\[
	\sum_{C: C \subset F, C \in \C_d(f,\lambda)} w(C) \ge 1
	\]
	for all $F \in \ff$.
\end{definition}

\begin{definition}
	Given $\lambda >0$ and a finite family $\ff \subset \C_d$, we define the \emph{fractional $f$-transversal number} $\tau^*_{(f,\lambda)}(\ff)$ as the minimum $\sum_{C\in \C_d(f,\lambda)}w(C)$ over all functions $w:\C_d(f,\lambda) \to [0,1]$ such that
	\[
	\sum_{C: C \subset F, C \in \C_d(f,\lambda)} w(C) \ge 1
	\]
	for all $F \in \ff$.
\end{definition}

\begin{definition}
	Given $\lambda >0$ and a finite family $\ff \subset \C_d$, we define the \emph{fractional $f$-packing number} $\nu^*_{(f,\lambda)}(\ff)$ as the maximum $\sum_{F \in \ff} w(F)$ over all functions $w: \ff \to [0,1]$ such that
	\[
	\sum_{F: C \subset F, F \in \ff} w(F) \le 1
	\]
	for all $C \in \C_d(f,\lambda)$.
\end{definition}

\begin{lemma}\label{lemma-first}
	Let $f:\C_d \to \rr^+ \cup \{\infty\}$ be a monotone function that admits weak $\varepsilon$-nets for convex sets.  Then, $\tau_{(f,1-\varepsilon)}(\ff)$ is bounded by a function that depends only on $f,d, \varepsilon$, and $\tau^*_{(f,1-\frac{\varepsilon}{2})}(\ff)$.
\end{lemma}

\begin{proof}
	Let $w: \C_d (f, 1-\frac{\varepsilon}2) \to [0,1]$ be a function that realizes $\tau^*_{(f,1-\frac{\varepsilon}{2})}(\ff)$, and write $r = \tau^*_{(f,1-\frac{\varepsilon}{2})}(\ff)$.  Without loss of generality, we may assume that $w$ has finite support and takes only rational values.  Let $M$ be the common denominator of $w(K)$ for all $K \in \C_d(f,1-\frac{\varepsilon}2)$.  Let $\ttt$ be the family that is formed by the disjoint union of $M \cdot w(K)$ copies of $K$, for each $K \in \C_d(f,1-\frac{\varepsilon}2)$.  Now consider $\mathcal{K}$ a weak  $\frac{1}{r}$-net for $\ttt$, as in Definition \ref{definition-weak-net-quant} (the hidden constant $\lambda$ being equal to $1-\frac{\varepsilon}2$).
	
	Notice that for each $K \in \mathcal{K}$, we have that $f(K) \ge (1-\frac{\varepsilon}{2})^2 \ge 1-\varepsilon$.  Also, by the definition of $\tau^*$, for a set $F \in \ff$, there are at least $\frac{1}{r}|\ttt|$ sets in $\ttt$ contained in $F$.  Thus, by the definition of $\mathcal{K}$, there must be a set in $\mathcal{K}$ contained in $F$.  This implies that the indicator function of $\mathcal{K}$ satisfies the conditions in the definition of $\tau_{(f,1-\varepsilon)}(\ff)$.  Moreover, $|\mathcal{K}|$ depends only on $r,f, \varepsilon$, and $d$, as desired. 
\end{proof}

\begin{lemma}\label{lemma-second}
	Let $f:\C_d \to \rr^+ \cup \{\infty\}$ be a monotone function that admits a fractional Helly theorem and $p \ge q \ge F(f,d,\frac{\varepsilon}2)$.  Let $\ff$ be a finite family of convex sets such that for every $K \in \ff$ we have $f(K) \ge 1$ and out of every $p$ sets in $\ff$, there is a $q$-tuple $A$ such that $f(\cap A) \ge 1$.  Then $\nu^*_{(1-\frac{\varepsilon}{2})}(\ff)$ is bounded by a function that depends only on $f,p,q, \varepsilon$, and $d$.
\end{lemma}

\begin{proof}
	Let $w:\ff \to [0,1]$ be a function that realizes $\nu^*_{(1-\frac{\varepsilon}{2})}(\ff)$, and write $r = \nu^*_{(1-\frac{\varepsilon}{2})}(\ff)$.  We may assume without loss of generality that $w(F)$ is rational for all $F \in \ff$.  Let $w(F) = \frac{N_F}{M}$ where $M$ is the common denominator for all $w(F)$ with $F \in \ff$.  Let $\ff'$ be a family consisting of $N_F$ copies of $F$ for each $F \in \ff$, and $N = |\ff'|$.  Note that $\frac{N}{M} = \sum_{F \in \ff} \frac{N_F}{M} = r$.
	
	The family $\ff'$ satisfies that out of every $(p-1)(q-1)+1$ of its sets, there is a $q$-tuple $A$ such that $f(\cap A) \ge 1$.  This is because given $(p-1)(q-1)+1$ sets of $\ff'$, there are either $q$ copies of the same set or at least $p$ different sets of $\ff$.  In either case, we have a $q$-tuple satisfying the property.  Thus, there is a positive fraction of the $F(f,d,\frac{\varepsilon}2)$-tuples of $\ff'$ whose intersection $A$ satisfies $f(A) \ge 1$.  By the fractional Helly theorem, there must be a set $A_0$ with $f(A_0) \ge 1-\frac{\varepsilon}{2}$ contained in at least $\beta N$ sets of $\ff'$.  In other words,
	\[
	1 \ge \sum_{F \in \ff: A_0 \subset F} w(F) = \sum_{F \in \ff: A_0 \subset F} \frac{N_F}{M} \ge \frac{1}{M} \beta N = \beta r.
	\]
	
	This implies that $r \le \frac{1}{\beta}$, and $\beta$ is bounded by a function depending only on $p,q,f, \varepsilon$ and $d$.
\end{proof}

\begin{proof}[Proof of Theorem \ref{theorem-general-quantitative-pq}]
	By linear programming duality, $\tau_{(f,1-\frac{\varepsilon}{2})}^*(\ff) = \nu_{(f,1-\frac{\varepsilon}{2})}^*(\ff)$.  Thus, using the lemmata \ref{lemma-first} and \ref{lemma-second}, we have a bound on $\tau_{(f,1-\frac{\varepsilon}{2})}(\ff)$ that only depends on $p,q,f,d,$ and $\varepsilon$.  Notice that a function that realizes $\tau_{(f,1-\frac{\varepsilon}{2})}(\ff)$ is the indicator of the family $\ttt$ of sets we are looking for in the theorem, i.e. every set in $\ttt$ is large according to $f$, and every set in $\ff$ contains a set in $\ttt$.
\end{proof}

\section{Conditions for a fractional Helly theorem}\label{section-fractional-both}

In order to prove the Corollaries \ref{theorem-quant-pq}, \ref{theorem-vol-pq} and \ref{theorem-surface}, it suffices to show that the volume, surface area and the indicator of having $k$ points of $S$ satisfy the desired properties.  The first two conditions (admitting a Helly theorem and being approximable with inscribed polytopes) were discussed in section \ref{section-strong}.  Thus, it only remains to show the fractional Helly theorems.  We do this in Lemma \ref{lemma-fractional-general} below.

In order to prove this lemma we will use previous results regarding the convex floating body \cite{Schutt:1990tj}.  For a convex set $K$ with finite volume, we define $K(\vol,\varepsilon)$ as the set of points $x$ such that $\vol (H \cap K) \ge \vol(K)(1-\varepsilon)$ for all closed halfspaces $H \ni x$.  There are several results regarding the volume of the floating body \cite{Barany:2010cy}.  For sufficiently smooth bodies $K$ of unit volume, we have
\[
\vol(K(\vol,\varepsilon)) \ge 1-c\varepsilon^{2/(d+1)}
\]
where $c$ is a constant depending only on the dimension (see \cite{leich}).  This result is described in \cite{Barany:2010cy}, where it is mentioned that it was proved by Buchta, Gruber and M\"uller but only appears as a private communication.  The result holds for sufficiently smooth convex sets, and the constant $c$ is maximized by ellipsoids.  However, by standard approximation results, we can see that the bound above extends to all convex sets.  The reason why this bound is presented for sufficiently smooth bodies is that in that case there is a matching upper bound.

Note that if $K$ is a convex set of unit volume and $K'$ is another convex set such that $\vol(K \cap K') \ge 1-\varepsilon$, then $K(\vol, \varepsilon) \subset K'$.  There is no reason why the floating body is unique to the volume, and indeed we can introduce the following definition

\begin{definition}\label{definition-floating-body}
	Given a monotone function $f:\C_d \to \rr^+ \cup \{\infty\}$, we say that there is a \emph{floating body} for $f$ if for every $1 > \varepsilon > 0$, there is a $\delta = \delta(\varepsilon) > 0$ such that any convex set $K$ with $f(K) < \infty$, there is a convex set $K(f,\varepsilon)$ such that 
	\begin{itemize}
		\item $f(K(f,\varepsilon)) \ge (1- \delta)f(K)$ and
		\item for any convex set $A$, if $f(A \cap K) \ge (1-\varepsilon)f(K)$, then $K(f,\varepsilon) \subset A$.
	\end{itemize}
	Morevoer, we require that $\delta (\varepsilon) \to 0$ if $\varepsilon \to 0$.  Throughout the rest of this section we will assume that $f$ and $d$ are fixed.  Even though $\delta (\varepsilon)$ depends on them, it will be useful to denote it this way.
\end{definition}

We first show how some general properties of $f$ are enough to prove the existence of a floating body, which the reader can check are satisfied for the surface area.  We then show how the existence of a floating body and a Helly theorem are enough to prove a fractional Helly theorem and a colorful Helly theorem.

\begin{definition}
	We say a function $f:\C_d \to \rr^+ \cup \{\infty\}$ is
	\begin{itemize}
		\item \emph{homogeneous} if there is a constant $k$ such that for all $\alpha >0$, we have that $f(\alpha K ) = \alpha^k f(K)$;
		\item \emph{strictly monotone} if $f$ is monotone and for every convex set $B$ with $0<f(B) < \infty$, if $A \subset B$ is a convex set such that the closure of $A$ is different from the closure of $B$, then $f(A) < f(B)$;
		\item \emph{well-defined}, if $f(\emptyset) = 0$ and $f(B) \in \{0,\infty\}$ for any unbounded convex set $B$.
	\end{itemize}
\end{definition}

\begin{theorem}
	Let $f:\C_d \to \rr^+ \cup \{\infty\}$ be a continuous, strictly monotone, homogeneous, and well-defined function.  Then, there is a floating body for $f$.
\end{theorem}

Note that our condition of continuity over convex sets is taken under the topology induced by the Hausdorff metric.

\begin{proof}
	Since $f$ is homogeneous, well-defined and continuous, it is enough to show the existence of floating bodies for sets $K$ which are compact and contained in the closed ball of radius one around the origin.  Let us show that for a fixed $K$ there is a $\delta=\delta(K,\varepsilon)$ and a floating body $K'=K(f,\varepsilon)$ satisfying the conditions of Definition \ref{definition-floating-body} with $f(K') = (1- \delta(K,\varepsilon))f(K)$.  We may assume that $\infty > f(K) > 0$ without loss of generality.	
	
	For every direction $v$, consider a $v$-halfspace a set of the form $\{x \in \rr^d: \langle x, v \rangle \le \alpha\}$ for some real $\alpha$, where $\langle \cdot, \cdot \rangle$ denotes the usual dot propduct.  Let $H'_v$ be the containment-minimal $v$-halfspace such that $f (H'_v \cap K) \ge (1- \varepsilon)f(K)$.  Note that $H'_v$ exists and is unique since $f$ is well-defined, strictly monotone, and continuous.  We define $K' = \cap_{v} H'_v$ and $\delta (K, \varepsilon) = 1- \frac{f(K')}{f(K)}$.
	
	Let $K'' \subset K$ be any compact convex set such that $f(K'') = 1- \varepsilon$.  If $K' \not\subset K''$, there must be a closed halfspace which strictly separates a point of $K'$ from all of $K''$.  However, if $v$ is the direction defining this hyperplane, it would contradict $K' \subset H'_v$.

	Let us show that $\delta(K, \varepsilon) \to 0$ as $\varepsilon \to 0$. If this was not true, there would be an $\alpha >0$ such that $\delta(K, \varepsilon) < (1- \alpha)f(K)$ for all $\varepsilon >0$.  Since $K(f,\varepsilon_1) \subset K(f, \varepsilon_2)$ if $\varepsilon_1 > \varepsilon_2$, we can consider the convex set $K_0 = \cup_{\varepsilon>0} K(f,\varepsilon)$.  By continuity of $f$, we have that $f(K_0) \le 1-\alpha$ and $K_0 \subset K$.  Thus, since $f$ is strictly monotone, there must be a $v$-halfspace $H_0$ such that $K_0 \subset H_0$ and $H_0$ does not contain the closure of $K$.
	
	Let $K_1 = K \cap H_0$.  It is clear that $K_0 \subset K_1 \subset K$, and the closure of $K_1$ and $K$ are different, so $f(K_1) = (1-\beta) f(K)$ for some $\beta > 0$.  Notice then that $K(f,\beta / 2) \not\subset K_1$, contradicting the fact that $K(f, \beta /2) \subset K_0$.  Thus $\delta(K, \varepsilon) \to 0$ as $\varepsilon \to 0$.
	
	Let $\delta(\varepsilon)$ be the supremum of $\delta(K, \varepsilon)$ over all compact $K$ with $f(K) \le 1$ contained in the closed unit ball centered at the origin.  Let us show that $\delta(\varepsilon) \to 0$ as $\varepsilon \to 0$.  If this was not the case, there would be a sequence of pairs $(K_1, \varepsilon_1), (K_2, \varepsilon_2), \ldots$ such that $\varepsilon_n \to 0$ and $K_n(f,\varepsilon_n) \not\to 0$.  However, since the space of sets we considered is compact under the Hausdorff metric, there would be a set $K^*$ such that $K_n \to K^*$.  Since $f$ is a continuous function and we are working with a compact metric space, it is uniformly continuous.  Thus, we can use the constructed sequence to show that $f(K^*(f,\varepsilon_n)) \not\to 0$ as $n \to \infty$, a contradiction.  
\end{proof}

\begin{lemma}\label{lemma-fractional-general}
	Let $f:\C_d \to \rr^+ \cup \{\infty\}$ be a well-defined, continuous, strictly monotone function that admits a Helly theorem and such that there is a floating body for $f$.  Then $f$ admits a fractional Helly theorem.  Moreover
	\[
	F(f,d,\varepsilon) \le H(f,d,\delta^{-1}(\varepsilon))
	\]
	where $\delta$ is the parameter induced by the floating body for $f$.
\end{lemma}

The reader may notice that any function that is well-defined, continuous and homogeneous allows for a Helly theorem, using a compactness argument such as the one in the last step of the proof above.  The same argument also shows that $f$ can be approximated by inscribed polytopes.

\begin{proof}[Proof of Lemma \ref{lemma-fractional-general}]
	For a convex set $K$, we refer to $f(K)$ as its \emph{size}.  In order to prove this lemma, it suffices to show that, given $\alpha >0$ and $\ff$ a finite family of $n$ convex sets in $\rr^d$, if an $\alpha$-fraction of the $H(f, d, \varepsilon)$-tuples are intersecting, then there is a set of size at least $1-\delta(\varepsilon)$ contained in a positive fraction $\beta$ of the sets in $\ff$.  For simplicity, let $h=H(f, d, \varepsilon)$.
	
	We may assume that the sets in $\ff$ are bounded.  Let $v$ be a direction.  We consider a $v$-halfspace to be a set of the form $\{x: \langle x, v\rangle \le \alpha\}$ for some real $\alpha$.  For each $(h-1)$-tuple $B=\{F_1, F_2, \ldots, F_{h-1}\}$ such that $f(\cap B) \ge 1$, let $H_B$ be the $v$-halfspace such that $f((\cap B)\bigcap H_B)=1$.  We denote this intersection by $K_B$.

Now consider an $h$-tuple $A$ of $\ff$ such that $f(\cap A) \ge 1$.  Among its $(h-1)$-tuples, there must be one, call it $A'$, such that $H_{A'}$ is containment-maximal.  It is clear that if we add $H_{A'}$ to $A$, in the resulting family the intersection of any $h$ sets has size at least one.  By the definition of $h$ we have that the intersection of this whole family has size at least $1-\varepsilon$.  However, this implies that the set not in $A'$ contains $K_{A'}(f,\varepsilon)$.

For each $h$-tuple $A$ such that $f (\cap A) \ge 1$, let $A'$ be one of its $(h-1)$-tuple with containment-maximal $H_{A'}$.  If a positive fraction of the $h$-tuples satisfy the condition of the problem, then a simple counting argument shows that there must be an $(h-1)$-tuple $M$ which was assigned to at least $\beta n$ different $h$-tuples, for some positive $\beta$ not depending on $n$.  Thus at least $\beta n$ sets contain $K_{M}(f,\varepsilon)$, as desired.
\end{proof}

\begin{corollary}\label{lemma-fractional-volume}
The function $f(\cdot) = \vol (\cdot)$ admits a fractional Helly theorem.  Moreover, for any fixed dimension $d$, we have that $F (\vol, d, \varepsilon) = O\left(\varepsilon^{-(d^2-1)/4}\right)$.	
\end{corollary}

\begin{corollary}\label{lemma-fractional-surface}
The function $f(\cdot) = \surface(\cdot)$ admits a fractional Helly theorem.
\end{corollary}
\begin{proof}
It is clear that the surface area function is homogeneous and well-defined.  It remains to prove that it is strictly monotone.  Suppose that $B_0$ is a convex set with $0 < \surface(B_0) < \infty$, and consider $A\subset B_0$ such that the closures of $A$ and $B_0$ are distinct.  Then, there exists a point $b$ of $B_0\setminus A$ and a halfspace $H$ such that $A\subset H$ and $b\not\in H$. Take $B_1=B_0\cap H$.  Note that the surface area of $B_1$ is strictly less than that of $B_0$.

If the closures of $B_1$ and $A$ are the same, then we are done.  Otherwise, we define $B_2$ from $B_1$ similarly.  By proceeding in this manner for $n$ steps, we attain either $B_n$ and $A$ with the same closure or else $B_n$ an arbitrarily good approximation of $A$, by a simple convexity argument.  Because $\surface(B_n)<\surface(B_{n-1})$ for every $n$, we conclude that $\surface(A)<\surface(B_0)$, as desired.
\end{proof}

By contrast, the diameter function is not strictly monotone, and actually fails to have a floating body.

\begin{lemma}
Let $S \subset \rr^d$ be a discrete set and $k$ a positive integer such that $\h_k (S) < \infty$. Then, the function $f(\cdot)$ which is the indicator of the property ``having at least $k$ points of $S$'' admits a sharp fractional Helly theorem.  Moreover, for any fixed dimension $d$, we have that $F (f, d, 0) \le \h_k (S)$.	
\end{lemma}

\begin{proof}
	The proof is equivalent to the one of Lemma \ref{lemma-fractional-general}.  In this case, $v$ must be chosen such that it is not orthogonal to any segment with endpoints in $S$.  Since $S$ is discrete, and thus countable, this is always possible.  In this case, following the notation above, $K_{M_0}$ is a subset of exactly $k$ points of $S$, and the same arguments follow.
\end{proof}

\section{Colorful Helly for continuous functions}\label{section-colorful-volumetric}

We show how the floating bodies defined in the previous section imply a colorful Helly theorem.  When adapted to the volume, this yields an essentially different proof of the volumetric version of Helly's theorem.  The first proof is given in \cite{DeLoera:2015wp}.

\begin{theorem}\label{theorem-vol-colorful}
	Let $f:\C_d \to \rr^+ \cup \{\infty\}$ be a well-defined, continuous, monotone function that admits a Helly theorem and such that there is a floating body for $f$.
	Let $h = H(f,d, \varepsilon)$, and $\ff_1, \ff_2, \ldots, \ff_h$ be finite families of convex sets in $\rr^d$, considered as color classes.  Suppose that the intersection of every colorful choice $F_1 \in \ff_1, \ldots, F_h \in \ff_h$ has size at least one under $f$.  Then, there is a color class $\ff_i$ for which
	\[
	f (\cap \ff_i) \ge 1-\delta(\varepsilon).
	\] 
\end{theorem}

\begin{proof}
	We follow the same technique as in the proof of Lemma \ref{lemma-fractional-general}.  For a convex set $K$, we refer to $f(K)$ as its \emph{size}.   We may assume without loss of generality that the sets in $\ff_1, \ldots, \ff_h$ are bounded.
	
	Consider $v$ a direction.  We consider a $v$-halfspace to be a set of the form $\{x: \langle x, v\rangle \le \alpha\}$ for some real $\alpha$.  For each colorful $(h-1)$-tuple $B=\{F_1, F_2, \ldots, F_{h-1}\}$ (i.e. each $F_i$ is in a different color class) we have $f(\cap B) \ge 1$.  Let $H_B$ be the $v$-halfspace such that $f((\cap B)\bigcap H_B)=1$.  We denote this intersection by $K_B$.
	
	Now consider a colorful $h$-tuple $A$ of $\ff$.  We know that $f(\cap A)\ge1$.  Among its $(h-1)$-tuples, there must be one, call it $A'$, such that $H_{A'}$ is containment-maximal.  It is clear that if we add $H_{A'}$ to $A$, in the resulting set the intersection of any $h$ sets has size at least one.  By the definition of $h$ we have that the intersection of this whole family has size at least $1-\varepsilon$.  However, this implies that the set not in $A'$ contains $K_{A'}(f,\varepsilon)$.
	
	Now let $B$ be a colorful $(h-1)$-tuple with containment-maximal $H_B$ over all possible colorful $(h-1)$-tuples.  Let $\ff_i$ be the color class that does not have a set in $B$.  The observations above imply that
	\[
	K_B(f, \varepsilon) \subset \bigcap \ff_i,
	\]
	finishing the proof.
\end{proof}

\begin{corollary}
	Let $h = H(\vol,d, \varepsilon)$, and $\ff_1, \ff_2, \ldots, \ff_h$ be finite families of convex sets in $\rr^d$, considered as color classes.  Suppose that the intersection of every colorful choice $F_1 \in \ff_1, \ldots, F_h \in \ff_h$ has volume at least one.  Then, there is a color class $\ff_i$ for which
	\[
	\vol (\cap \ff_i) \ge 1-c\varepsilon^{2/(d+1)}
	\] 
	for some constant $c$ depending only on the dimension.
\end{corollary}

\section{Remarks}\label{section-remarks}

It is unclear if the constants needed for the volumetric Helly theorem, the colorful volumetric Helly theorem and the fractional volumetric Helly theorem should be different or not.  In particular

\begin{problem}
	Is it true that $F(\vol, d, \varepsilon) > H(\vol, d, \varepsilon)$?	
\end{problem}

The known results where the fractional Helly number is different from the Helly number all require checking smaller subfamilies for the fractional version.  It would be interesting to have fractional Helly results which require stronger conditions than their Helly counterpart.

As far as the authors know, there are no examples where a colorful Helly theorem requires larger family sizes than its monochromatic counterpart.  There are currently two different proofs of the colorful volumetric Helly theorem which require $O(\varepsilon^{-(d^2-1)/4})$ color classes, as opposed to the $\Theta (\varepsilon^{-(d-1)/2})$ for the monochromatic version.

\begin{problem}
	For the colorful volumetric Helly theorem, are $O(\varepsilon^{-(d-1)/2})$ color classes sufficient?	
\end{problem}

\section{Acknowledgments}

The authors would like to thank Jes\'us De Loera for stimulating discussions on the topic.  We would also like to thank the University of Michigan, Northeastern University, and the Massachusetts Institute of Technology for providing space and resources to carry out this research.  D.R.~was supported by the National Science Foundation Graduate Research Fellowship under Grant No.~1122374. 

\bibliographystyle{amsalpha}

\bibliography{references.bib}

\noindent David Rolnick \\
\textsc{
Mathematics Department \\
Massachusetts Institute of Technology \\
Cambridge, MA 02139
}\\[0.1cm]

\noindent Pablo Sober\'on \\
\textsc{
Mathematics Department \\
Northeastern University \\
Boston, MA 02445
}\\[0.1cm]

\noindent \textit{E-mail addresses: }\texttt{drolnick@math.mit.edu, p.soberonbravo@neu.edu}

\end{document}